\documentclass[12pt]{amsart}

\usepackage[all]{xy}

\usepackage{amsmath,amssymb,latexsym, amsfonts}
\usepackage{verbatim}
\usepackage{times}
\usepackage{mathbbol}
\usepackage{hyperref}

\def\id{{\rm id}}\def\End{{\rm End}}

\newtheorem{proposition}{Proposition}
\newtheorem{corollary}{Corollary}
\newtheorem{theorem}{Theorem}
\newtheorem{lemma}{Lemma}
\theoremstyle{definition}
\newtheorem{definition}{Definition}
\newtheorem{remark}{Remark}
\newtheorem{question}{Question}
\newenvironment{bemerkung}
  {\pushQED{\qed}\remark}
  {\popQED\endremark}
\newtheorem{example}{Example}
\newenvironment{beispiel}
  {\pushQED{\qed}\example}
  {\popQED\endexample}

\begin{document}
\title{A universal enveloping algebra for
cocommutative rack bialgebras}
\author{Ulrich Kr\"ahmer and Friedrich Wagemann}
\begin{abstract}
We construct a bialgebra object in the category of linear 
maps $\mathcal {LM}$
from a cocommutative rack bialgebra. The
construction does extend to some
non-cocommutative rack bialgberas, as is
illustrated by a concrete example.  
As a separate result, we show that the Loday complex 
with adjoint coefficients embeds into the rack bialgebra 
deformation complex for the rack 
bialgebra defined by a Leibniz 
algebra. 
\end{abstract}

\maketitle

\tableofcontents

\section*{Introduction}
A \emph{shelf} is a set $X$
with a composition $(x,y)\mapsto x\lhd y$ which satisfies the 
self-distributivity relation
$(x\lhd y)\lhd z=(x\lhd z)\lhd(y\lhd z)$.
More generally, this condition makes sense for 
coalgebras in braided
monoidal categories, as was observed by
Carter-Crans-Elhamdadi-Saito \cite{CCES} and further
studied by Lebed  
\cite{Leb}. In the present article, we will focus on
shelves in the category of vector spaces with the
tensor flip as braiding:
 
\begin{definition}
A \emph{linear shelf} is a 
coassociative coalgebra 
$(C,\triangle)$ together with a morphism of coalgebras 
$$
	C \otimes C \rightarrow C,\quad
 	(x,y)\mapsto x \lhd y
$$ 
that satisfies
\begin{equation}\label{selfdistributive}
	(x\lhd y)\lhd z=
	(x\lhd z_{(1)})\lhd(y\lhd z_{(2)})
	\quad \forall	x,y,z \in C.
\end{equation}
A counital and coaugmented 
linear shelf $(C,\triangle,\lhd,\epsilon,1)$ for which 
$$
	x\lhd 1=x,\quad 
	1\lhd x=\epsilon(x)1,\quad
	\epsilon (x \lhd y) = 
	\epsilon(x)\epsilon(y) 
$$ 
holds for all
$x,y \in C$ will be called a
\emph{rack bialgebra}.
\end{definition}

Here and elsewhere, 
all vector spaces, coalgebras etc will be
over a field $k$, and we use Sweedler's notation 
$\triangle (z)=z_{(1)}
\otimes z_{(2)}$ for coproducts.

If $C$ is spanned by primitive elements
(together with the coaugmentation), the
definition of a rack bialgebra reduces
to that of a Leibniz algebra \cite{LP}. 
Lie racks provide another 
natural source of examples 
with a rich structure theory and
applications in the deformation quantisation of duals of Leibniz
algebras, see 
\cite{ABRW}, \cite{ABRWI}. 

An important step in the theory of
Leibniz algebras was the definition of
their universal enveloping algebras
\cite{LPneu}. Here, we extend this
construction to
all cocommutative rack bialgebras. The
result is a cocommutative bialgebra
$U(C)$ in the category of vector spaces,
see
Theorem~\ref{theorem_braided_Leibniz}
below. 

One application of universal enveloping
algebras is to express co\-ho\-mo\-logy
theories as derived functors.  
Our motivation for the present article
was the article \cite{CCES}. Therein,
the authors develop a deformation theory of 
linear shelves. 
To this end, they defined cohomology groups 
$H^n_\mathrm{sh}(C,C)$ for $n \le 3$. 
However, it remained an open question
how to extend this to a fully fledged
cohomology theory including an
interpretation as derived functors in an
abelian category. In \cite{Alissa}, this
was further studied in the special case
where $\lhd$ is also associative.

The present article is meant as a first
step towards a possible answer to this
question. Our initial goal was to apply 
our results from \cite{KW}:
therein, we constructed examples of rack
bialgebras from Hopf
algebras in Loday-Pirashvili's
category of linear maps $\mathcal {LM}$ \cite{LP}. 
Therefore, Gerstenhaber-Schack cohomology \cite{GS} 
in the tensor category $\mathcal{LM}$ could be used to define 
$H^n_\mathrm{sh}(C,C)$ if all linear
shelves arose in this way. 

Let us describe the content of the present article
section by section.
Section~1  
contains preliminaries on coalgebras and 
points out that counitisation does not
provide an equivalence between linear
shelves and rack bialgebras. 
Section 2 recalls from \cite{KW} the 
construction of rack bialgebras 
from Hopf algebras. 
Section 3 introduces the notion of a
Yetter-Drinfel'd rack which guides 
the construction of $U(C)$ in
the main subsequent Section 4. 
Therein, we define $U(C)$ and establish
its universal property. We also
reformulate these results in terms of 
a bialgebra object in 
${\mathcal L}{\mathcal M}$ and give some
examples. The construction of $U(C)$
does extend also to some
non-cocommutative rack bialgebras. This
is demonstrated with an explicit example 
in Section 5. 
In Section 6, we describe this example
as a deformation of a cocommutative rack
bialgebra. 
Furthermore,  
the Loday complex 
with adjoint coefficients is embedded 
into the deformation complex
of the rack bialgebra associated to a
Leibniz algebra. The article 
concludes with an outlook and some open
questions.   

\subsection*{Acknowledgements}
We thank Alissa Crans for interesting
discussions that motivated us to work on
the topic of this paper.




\section{Rack bialgebras and linear shelves
are not equivalent}
One might expect that linear shelves and rack
bialgebras are related by
counitisation. We begin by
pointing out that this is not the case, so
when developing cohomology or deformation
theories, one must be clear 
which of the two structures one is studying.   
The construction from \cite{KW}
inevitably yields rack bialgebras, 
hence these are the objects we will
focus on afterwards. 

More precisely, recall that if a coalgebra 
has a counit 
$$ 
	\epsilon \colon C \rightarrow k,\quad 
	\epsilon(c_{(1)})c_{(2)}=\epsilon(c_{(2)})c_{(1)}=c
\quad\forall c \in C,
$$
and is in addition coaugmented,
i.e.~has a distinguished group-like element  
$$
	1 \in C,\quad 
	\triangle (1)=1 \otimes 1, 
$$
then the vector space 
$\check C:=\mathrm{ker}\, \epsilon$ becomes a
coalgebra with coproduct 
$$
	\check \triangle (c):=\triangle(c) -
	1 \otimes c-c \otimes 1, 
$$ 
or, in Sweedler notation, 
$$
	c_{{(\check 1)}} \otimes c_{{(\check 2)}} =
	c_{(1)} \otimes c_{(2)}-1 \otimes c 
	-c \otimes 1.
$$
Furthermore, the map 
$c \mapsto (\epsilon(c)1,c-\epsilon (c)1)$  
splits $C$ 
canonically into a direct sum $C=k 1 \oplus 
\check C$. This shows that 
$C \mapsto \check C$ is 
an equivalence
between the category of counital and coaugmented
coalgebras and the category of all
coalgebras, with
inverse $C \mapsto \hat C := k \oplus C$ and
$\hat\triangle(x) = \triangle(x) + 1 \otimes
x + x \otimes 1$. 

If $C$ is a rack bialgebra, then
$\lhd$ does restrict to $ \check C = 
\mathrm{ker}\, \epsilon$, but it is in general
\emph{not} self-distributive with respect to 
$\check \triangle$, so $\check C$ does \emph{not} become a
linear shelf in its own right. 
Conversely, if $C$ is a linear shelf, then
$\hat C$ is in general not a rack bialgebra
with respect to $\hat\triangle$.  
In fact, we have:

\begin{proposition}
The counitisation functor does not lift
to an equivalence from the category
of linear shelves to the category of rack
bialgebras. 
\end{proposition} 	 
\begin{proof}
Consider linear
shelves $C$ with vanishing coproduct
$\triangle=0$. The category of these has a zero
object, the shelf of vector space dimension 
$\dim_k(C)=0$, and a unique simple object,
the shelf of dimension $ \dim_k(C)=1$ with 
$\lhd = 0$. For all other objects $C$
there is a morphism of shelves
$C \rightarrow D$ with
nonzero kernel and nonzero image $D \neq 0$.  
Indeed, the quotient vector space
$C/\mathrm{im}\,\lhd$ is a linear shelf 
with respect to $\triangle=\lhd=0$, and  
the canonical projection $ C
\rightarrow C/\mathrm{im}\, \lhd$ 
is a morphism of linear
shelves. 
Its kernel vanishes if and
only if $\lhd=0$, that is, if $C$
is a direct sum of 1-dimensional shelves with
trivial coproduct and shelf product. In
this case, the canonical projection onto 
any quotient vector space $D$
of dimension $\dim_k(C)-1$ has the
desired properties. If instead $ \lhd
\neq 0$, then we can take the quotient 
$C \rightarrow D:=C/\mathrm{im}\, \lhd $
itself: for $\triangle = 0$, the
self-distributivity condition reads 
$(x \lhd y) \lhd z=0$ for all $x,y,z \in
C$; in particular, $ \mathrm{im}\, \lhd
\subseteq \bigcap_{x \in C} 
\mathrm{ker}\, (- \lhd x) \neq C$, 
hence $D \neq 0$.

In contrast, the shelf condition on the
counitisation 
$\hat C$ of a coalgebra $C$ with vanishing
coproduct says precisely that the subspace $C
\subset \hat C$ is a Leibniz algebra. In
particular, this means that among the rack
bialgebras with this underlying coalgebra
structure, we have the (counitisations of) 
all simple Lie algebras, and these 
do not admit any nontrivial proper quotients. 
\end{proof}

However, there are some subclasses of linear
shelves which do admit counitisations that
can be turned functorially into rack
bialgebras. The most important one is
obtained from linearised shelves 
spannned by group-like elements:

\begin{beispiel}\label{erstesbsp}

Assume that $(C,\triangle,\lhd)$ is a linear
shelf with a vector space basis $G$ 
consisting of group-like
elements. Then $\{1,1+x \mid x \in G\}$ is a
vector space basis of the counitisation $\hat{C}$. 
Now define a new rack product $\blacktriangleleft$
on $\hat{C}$ by
$$(1+x)\blacktriangleleft(1+y):=1+(x\lhd y).$$ 
The self-distributivity for $\blacktriangleleft$ follows immediately from the selfdistributivity of $\lhd$ and thus $\hat{C}$ becomes a rack bialgebra. 
\end{beispiel}

\section{From Hopf algebras to rack bialgebras}

Let $(H,\triangle_H,\epsilon_H,\mu_H,1_H,S_H)$ be a 
Hopf algebra over $k$. Then $ \mathrm{ker}\,
\epsilon_H$ is a (right right) 
Yetter-Drinfel'd module with respect
to the right adjoint action
$$
    h\lhd h' :=S(h'_{(1)})h h'_{(2)}
$$
and the right coaction
$$
	h \mapsto h_{(0)} \otimes 
	h_{(1)} := h_{(1)} \otimes 
	h_{(2)} -1 \otimes h.
$$
If $C \subset \mathrm{ker}\, \epsilon_H$ is invariant
under the adjoint action, then 
$\lhd$ restricts to a product 
$C \otimes C \rightarrow C$ which satisfies 
(\ref{selfdistributive}). However, 
$z_{(1)} \otimes z_{(2)}$ is the coproduct in $H$, and
in general, this does not restrict to 
$C$. One situation where this approach leads to linear
shelves is the following:

\begin{proposition}\label{husten}
Let $H$ be a cocommutative Hopf algebra and 
$C \subset \mathrm{ker}\, \epsilon_H$ be a
Yetter-Drinfel'd submodule. Then 
$\hat C := k 1_H \oplus C$ is a rack bialgebra with
respect to the adjoint action $\lhd$ and the restriction of $\triangle_H$ to $\hat C$.  
\end{proposition}

\begin{proof}
As $C$ is a (right) subcomodule, we have $h_{(0)}\otimes h_{(1)}\in C\otimes H$. But 
$h_{(0)}\otimes h_{(1)}=h_{(1)}\otimes h_{(2)}-1\otimes h$, thus we conclude that 
$h_{(1)}\otimes h_{(2)}=h_{(0)}\otimes h_{(1)}+1\otimes h\in C\otimes H$. But by cocommutativity,
this implies that  $h_{(1)}\otimes h_{(2)}=h_{(2)}\otimes h_{(1)}\in H\otimes C$. In conclusion,
$$h_{(1)}\otimes h_{(2)}\in C\otimes H\cap H\otimes C=C\otimes C,$$
i.e. $C$ is stable under the coproduct of $H$. 

The adjoint action is a morphism of coalgebras thanks to cocommutativity:
\begin{align*}
& \quad (a \lhd b)_{(1)}
	\otimes (a \lhd b)_{(2)} \\
& =
(S(b_{(1)})ab_{(2)})_{(1)}
	\otimes(S(b_{(1)})ab_{(2)})_{(2)} \\
&= (S(b_{(1)}))_{(1)}a_{(1)}(b_{(2)})_{(1)}
	\otimes
	(S(b_{(1)}))_{(2)}a_{(2)}(b_{(2)})_{(2)}  \\
&= S(b_{(3)})a_{(1)}b_{(2)}\otimes 
	S(b_{(1)})a_{(2)} b_{(4)}  \\
&= S(b_{(1)})a_{(1)}b_{(2)} 
	\otimes S(b_{(3)})a_{(2)} b_{(4)} \\
&= (a_{(1)} \lhd b_{(1)})) \otimes  
	(a_{(2)} \lhd b_{(2)}),
\end{align*}
where we have used cocommutativity in the last step.   

As mentioned before, the self-distributivity is the only property which works independently of the cocommutativity of $C$. Indeed, on the one hand, we have:
$$(a\lhd b)\lhd c = S(c_{(1)})S(b_{(1)})ab_{(2)}c_{(2)} $$
And on the other hand, we have:
\begin{align*}
& \quad 
	(a\lhd c_{(1)})\lhd(b\lhd c_{(2)}) \\
&= S(S(c_{(3)})_{(1)}b_{(1)}(c_{(4)})_{(1)})(S(c_{(1)})ac_{(2)})(S(c_{(3)})_{(2)}b_{(2)}(c_{(4)})_{(2)}) \\
&= S(S(c_{(4)})b_{(1)}c_{(5)})(S(c_{(1)})ac_{(2)})(S(c_{(3)})b_{(2)}c_{(6)}) \\
&= S(c_{(5)})S(b_{(1)})S^2(c_{(4)}) S(c_{(1)})ac_{(2)}S(c_{(3)})b_{(2)}c_{(6)} \\
&= S(c_{(3)})S(b_{(1)})S(c_{(2)}) S(c_{(1)})ab_{(2)}c_{(4)} \\
&= S(c_{(1)})S(b_{(1)})ab_{(2)}c_{(2)}
\qedhere
\end{align*} 
\end{proof}

\begin{beispiel}  \label{example_shelf}
If $X$ is a shelf in the category of sets 
\cite{FennRourke}, then as discussed in
Example~\ref{erstesbsp}, the
counitisation $C=\widehat{kX}$ of its linearisation 
becomes a rack bialgebra in which all $x \in X$
are group-like. 
Observe that this construction differs
slightly from the construction in
\cite{CCES}, Section 3.1.
\end{beispiel} 
 
\begin{beispiel}  \label{example_Leibniz_algebra}
Given a (right) Leibniz algebra ${\mathfrak h}$, the $k$-vector space 
$C:=k\oplus{\mathfrak h}$ becomes 
a rack bialgebra by extending the bracket $[x,y]=:x\lhd y$ to a shelf product on all of 
$k\oplus{\mathfrak h}$. More precisely, we endow first of all $C$ with a coproduct requiring that 
all elements of ${\mathfrak h}$ are primitive, $\triangle(1)=1\otimes 1$,
$\epsilon(x)=0$ for all $x\in{\mathfrak h}$ and $\epsilon(1)=1$. Then put for all 
$x,y\in{\mathfrak h}$ $x\lhd y=[x,y]$, $1\lhd
x=\epsilon(x)1$ and $x\lhd 1=x$. 
This gives a rack bialgebra. 
\end{beispiel}

\section{Yetter-Drinfel'd racks}
The question arises which rack
bialgebras
arise as in Proposition~\ref{husten}.
Just as the cocommutativity of $H$
was therein a sufficient, but not a necessary
assumption, the construction of a
bialgebra that we carry out now can 
also be applied to certain noncocommutative 
rack bialgebras. Hence we consider the
following general setting adapted from 
\cite[Proposition~5.5]{KW}:

\begin{definition}\label{def-ydr}
Let $H$ be a bialgebra. A
\emph{Yetter-Drinfel'd rack} over $H$ 
is a rack bialgebra $C$ together with a right
$H$-module structure $\cdot \colon 
C \otimes H \rightarrow C$ rendering 
$C$ an $H$-module coalgebra,
and a morphism 
$ q \colon C \rightarrow H$ of counital
coaugmented coalgebras such that 
\begin{equation}\label{c}
	a \lhd b = a \cdot q(b)
\end{equation}
and
\begin{equation}\label{d}
	h_{(1)} q(a \cdot h_{(2)}) = 
	q(a)h 
\end{equation}
hold for all $h \in H$ and $a,b \in C$.
\end{definition}

In the cocommutative setting, we have, 
as the name suggests:

\begin{proposition}\label{proposition3}
Let $H$ be a cocommutative bialgebra and 
$C$ be a Yetter-Drinfel'd rack. 
Then $C$ becomes a 
Yetter-Drinfel'd module with respect to the
coaction 
$C \rightarrow C \otimes H$,
$$
	x \mapsto 
	x_{(0)} \otimes x_{(1)}:=
	(x_{(1)} - \epsilon (x_{(1)})) 
	\otimes q(x_{(2)}) +
	\epsilon (x) \otimes 1.
$$ 
\end{proposition} 
\begin{proof}
The above formula defines a right coaction, as it is 
constructed using the coproduct, the counit and the morphism of coalgebras $q$.  

Let us check the Yetter-Drinfel'd property. The two sides in the YD-property have three terms. Let us reason term by term. For the first term, we have for $x\in C$ and $h\in H$
\begin{multline*}
(x\cdot h_{(2)})_{(1)}\otimes h_{(1)} q(x\cdot h_{(2)})_{(2)}=\\
=x_{(1)}\cdot h_{(2)}\otimes h_{(1)} q(x_{(2)}\cdot h_{(3)})=\\
=x_{(1)}\cdot h_{(1)}\otimes q(x_{(2)})h_{(2)},
\end{multline*}
where we were able to apply the above Condition (\ref{d}) in the last step only thanks to cocommutativity. 

For the second term, we have
$$-1\otimes h_{(1)}q(x\cdot h_{(2)})=-1\cdot h_{(1)}\otimes q(x)h_{(2)}$$
thanks to Condition (\ref{d}) and $1\cdot a=\epsilon(a)1$. The third term is simply
$$
	\epsilon(x\cdot h_{(2)})1 \otimes h_{(1)}=
	\epsilon(x)1\cdot h_{(1)}\otimes h_{(2)},$$
which is simply true again by $1\cdot a=\epsilon(a)1$. 
\end{proof} 

\begin{bemerkung}
Note that for elements $h=q(c)$ 
in the image of $q$, the $H$-module coalgebra 
property $(x \cdot h)_{(1)} \otimes 
(x \cdot h)_{(2)} = 
x_{(1)} \cdot h_{(1)} \otimes 
x_{(2)} \cdot h_{(2)}$ 
is satisfied automatically 
by the fact that $\lhd$ is a 
morphism of coalgebras. 
Hence the $H$-module coalgebra condition
in Definition~\ref{def-ydr} can be omitted if $H$
is generated as an algebra by $
\mathrm{im}\, q$. 
\end{bemerkung}

If $H$ is a Hopf algebra (admits an antipode), then 
$\mathrm{ker}\, \epsilon $ is a
Yetter-Drinfel'd module with respect to the
right adjoint action, and (\ref{d}) and the 
coalgebra morphism condition on $q$ are
equivalent to 
$q|_C \colon C \rightarrow \mathrm{ker}\,
\epsilon $ being a morphism of
Yetter-Drinfel'd modules.

\section{The universal enveloping
algebra $U(C)$} 
\label{subsection_construction}
Given any rack bialgebra $C$, let 
$T=k \oplus \check C \oplus 
\check C^{\otimes 2} \oplus
\ldots$ denote the tensor algebra of 
$\check C = \mathrm{ker}\, \epsilon$ and 
$i \colon C \rightarrow T$ 
be the
canonical inclusion, which is the
identity on $\check C$ and maps the
distinguished group-like $1 \in C$ to 
the scalar $1 \in k=\check C^{\otimes
0}$. As we will also consider
the tensor product $T \otimes T$, we
denote the product in $T$ by $.$ rather
than $ \otimes $. 
By the universal
property of $T$, the linear map 
$$
	\check C \rightarrow T \otimes T,\quad
	x \mapsto i(x_{(1)}) 
	\otimes i(x_{(2)})
$$
extends uniquely to an algebra map 
$ \triangle_{T} \colon T \rightarrow 
T \otimes T$. The  
coassociativity of the coproduct in $C$ 
implies that 
of $\triangle_{T}$. That is, 
$T$ becomes a bialgebra and $i$ yields an
embedding of counital coaugmented 
coalgebras $C \rightarrow
T$. 

Using once more the 
universal property of $T$, the rack product
$$
	\lhd \colon \check C \to\End(C),\quad 
	x\mapsto (y\mapsto y\lhd x)
$$
extends
to an algebra 
homomorphism $T \to\End(C)$, so 
$C$ becomes a right $T$-module coalgebra such that  
$x \lhd y = x\cdot i(y)$.

However, $i$ does not turn $C$ into a
Yetter-Drinfel'd rack over $H=T$, as 
the commutativity relation (\ref{d}) is not
satisfied in general. Hence we add the 
relations manually: 
\begin{definition}
For any rack bialgebra $C$ 
we denote by $U(C)$ 
the symmetric algebra
of $C$ with respect to the  
vector space braiding 
$$
	\tau \colon 
	C \otimes C \rightarrow C \otimes C,
	\quad
	x \otimes y \mapsto 
	y_{(1)} \otimes x \lhd y_{(2)},
$$  
that is, $U(C):=T/J$ where $T=T(\check
C)$ and 
$$
	J :=
	\langle  i(y_{(1)}).i(x\lhd y_{(2)}))
	-i(x).i(y)\,|\,x,y\in C\rangle.
$$
We call $U(C)$ the \emph{universal
enveloping algebra} of $C$ and 
denote the canonical map $C
\rightarrow U(C)$ by $q$. 
\end{definition}

A key observation is that in case
$C$ is cocommutative, the
coproduct $ \triangle_{T}$ 
descends to $U(C)$:

\begin{lemma}
The ideal $J$ is also a coideal 
in case $C$ is cocommutative.
\end{lemma}

\begin{proof}
As $ \triangle_T $ is an algebra map, it is
sufficient to prove that the coproduct of a
generating element of $J$ belongs to 
$T \otimes J + J \otimes T$.
\begin{align*} 
& \quad \triangle_T(i(y_{(1)}).i(x\lhd y_{(2)})-
	i(x).i(y)) \\
&=
 i(y_{(1)}).i(x\lhd y_{(3)})_{(1)}\otimes i(y_{(2)}).i(x\lhd
y_{(3)})_{(2)}\\
& \quad -i(x_{(1)}).i(y_{(1)})\otimes i(x_{(2)}).i(y_{(2)})\\
&= i(y_{(1)}).i(x_{(1)}\lhd y_{(3)})\otimes
i(y_{(2)}).i(x_{(2)}\lhd y_{(4)})\\
& \quad -i(x_{(1)}).i(y_{(1)})\otimes 
i(x_{(2)}).i(y_{(2)})
\end{align*}
By cocommutativity, this is an element
of 
$T \otimes J + J \otimes T$:
\begin{align*}
& \quad i(y_{(1)}).i(x_{(1)}\lhd y_{(3)})\otimes
i(y_{(2)}).i(x_{(2)}\lhd
y_{(4)})\\
&\quad -i(x_{(1)}).i(y_{(1)})\otimes 
i(x_{(2)}).i(y_{(2)}) \\
& = i(y_{(3)}).i(x_{(1)}\lhd y_{(4)})\otimes
i(y_{(1)}).i(x_{(2)}\lhd
y_{(2)})\\
&\quad -i(x_{(1)}).i(y_{(1)})\otimes 
i(x_{(2)}).i(y_{(2)}) \\
& = \big(i(y_{(3)}).i(x_{(1)}\lhd
y_{(4)})-i(x_{(1)}).i(y_{(3)})\big)\otimes
i(y_{(1)}).i(x_{(2)}\lhd y_{(2)}) +\\
&\quad + i(x_{(1)}).i(y_{(3)})\otimes
i(y_{(1)}).i(x_{(2)}\lhd
y_{(2)})\\
&\quad -i(x_{(1)}).i(y_{(1)})\otimes
i(x_{(2)}).i(y_{(2)}) \\
&=  \big(i(y_{(3)}).i(x_{(1)}\lhd
y_{(4)})-i(x_{(1)}).i(y_{(3)})\big)\otimes
i(y_{(1)}).i(x_{(2)}\lhd y_{(2)}) +\\
&\quad + i(x_{(1)}).i(y_{(1)})\otimes\big(
i(y_{(2)}).i(x_{(2)}\lhd
y_{(3)})-i(x_{(2)}).i(y_{(2)})\big) \\
&\in J\otimes T + T\otimes J.\qedhere
\end{align*}
\end{proof}

Note that the action of $T$ 
on $C$ passes to an action of $U(C)$ on
$C$,
thanks to the self-distributivity of 
$\lhd$:

\begin{lemma}   \label{lemma_action}
For all $x,y,z\in C$, we have:
$$
	(x\cdot i(y))\cdot i(z)\,=
\,(x\cdot i(z_{(1)}))\cdot i(y \lhd z_{(2)})
$$
\end{lemma}

\begin{proof}
We have
\begin{align*}
& \quad (x\cdot i(z_{(1)}))\cdot i(y\lhd
z_{(2)}) \\
&= (x\lhd z_{(1)})\lhd(y\lhd z_{(2)}) \\
&= (x\lhd y)\lhd z \\
&= (x\cdot i(y))\cdot i(z).\qedhere
\end{align*}
\end{proof}

We thus arrive at the following
theorem, which realises $\check C$ as in 
\cite[Proposition~5.5]{KW} as a braided
Leibniz algebra:

\begin{theorem}    \label{theorem_braided_Leibniz}
The universal enveloping algebra 
$U(C)$ of a cocommutative rack bialgebra is 
canonically
a bialgebra, and $C$ becomes canonically 
a $U(C)$-Yetter-Drinfel'd rack.  
If furthermore $q_H \colon C \rightarrow H$ 
is any 
Yetter-Drinfel'd rack structure on $C$, 
then there exists a unique morphism 
of bialgebras $u \colon U(C) \to H$ such
that $u \circ q=q_H$ and hence 
\begin{equation}  \label{equivariance_property}
	x\cdot_{U(C)} s \,=\, x\cdot_H u(s)
\end{equation}
holds for all $x \in C,s \in U(C)$.
\end{theorem} 
\begin{proof}
The $U(C)$-Yetter-Drinfel'd rack
structure of $C$ has been established
already. 
By the universal property
of the tensor algebra, there 
exists a unique algebra homomorphism
$T(\check C) \to H$ such that 
$$
	c_1.\cdots.c_l \mapsto
	q_H(c_1) \cdots q_H(c_l).
$$
This is a morphism of coalgebras on the level of 
generators, and thus, by multiplicativity, 
a morphism of coalgebras, i.e. a morphism of 
bialgebras. 

As $C$ is a Yetter-Drinfel'd rack over
$H$, we have
$$
	q_H(x)q_H(y) = q_H(y_{(1)}) q_H(x \lhd 
	y_{(2)}),\quad
	x,y \in \check C.  
$$
Hence the bialgebra map induces a 
bialgebra morphism
$$
	u \colon U(C)\to H.
$$
The equality $u \circ q=q_H$
is true by construction.  
\end{proof}

\begin{corollary}\label{schnupfen}
If $C$ is a cocommutative rack bialgebra
and $U(C)$ is a Hopf algebra, then $C$
arises as in Proposition~\ref{husten}.
\end{corollary}
\begin{proof}
Indeed, just take $H=T \rtimes U(C)$ - 
taking the semidirect product is
necessary when $q \colon C \rightarrow
U(C)$ is not injective as in
Example~\ref{beispiel_Leibniz} below. 
\end{proof}

We now construct a bialgebra object 
in Loday-Pirashvili's category 
${\mathcal L}{\mathcal M}$ out of the
Yetter-Drinfel'd rack $C$. Recall that 
$\mathcal{LM}$ is the monoidal 
category of linear maps between vector
spaces with the so-called infinitesimal
tensor product as monoidal product, see
\cite{LP} for details. As explained in 
\cite{KW}, a bialgebra object 
in ${\mathcal L}{\mathcal M}$ consists
of a 
bialgebra $H$, an 
$H$-tetramodule $M$ and an $H$-bilinear 
coderivation $f \colon M\to H$. 
As is well-known \cite[Section~13.1.3]{KS}, any
Yetter-Drinfel'd module $V$ over a
bialgebra $H$ gives rise to a
tetramodule whose underlying vector
space is $H \otimes V$. Its actions and
coactions are given by
$$
	g (h \otimes v) g' :=
	ghg'_{(1)} \otimes 
	v \cdot g'_{(2)}, 
$$
$$
	(h \otimes v)_{(-1)} \otimes 
	(h \otimes v)_{(0)} \otimes 
	(h \otimes v)_{(1)} =
	h_{(1)} \otimes (h_{(2)} 
	\otimes v_{(0)}) \otimes 
	h_{(3)} v_{(1)} .
$$

Now we have:

\begin{theorem}  \label{theorem_bialgebra_object}
If $C$ is a cocommutative rack
bialgebra, then 
$$
	U(C) \otimes
	\check C \rightarrow U(C),\quad
	x \otimes y \mapsto x.q(y)  
$$ 
is canonically a bialgebra object 
in Loday-Pirashvili's category of linear maps 
${\mathcal L}{\mathcal M}$. 
\end{theorem} 
\begin{proof}
In light of
Proposition~\ref{proposition3}
and
Theorem~\ref{theorem_braided_Leibniz}, any
cocommutative rack bialgebra $C$ becomes
a Yetter-Drinfel'd module over 
$H=U(C)$. Furthermore, the
decomposition 
$C=k \cdot 1 \oplus \check C$ is a
direct sum of Yetter-Drinfel'd modules. 
Thus
$M:=U(C) \otimes \check C$ is a Hopf
tetramodule. That the linear map given
by $ s \otimes c \mapsto sq(c)$
is a coderivation and a bimodule map is
straightforwardly verified. 
\end{proof}

\begin{beispiel}
Note that the bialgebra $U(C)$ is not a
Hopf algebra in general, i.e.
does not necessarily have an antipode. 
For example, for $C=k \cdot 1 \oplus 
k \cdot g$ with 
$\triangle_C(g)=g\otimes g$ and 
$g \lhd g = g$, we
obtain for $U(C)$ a polynomial algebra in 
one group-like generator which does not 
have an antipode. This rack bialgebra $C$
is the counitisation of the 
linearisation of the conjugation
rack of the trivial group.
In general, if $C=\widehat{kX}$ 
for a rack $X$ as in
Example~\ref{example_shelf}, then 
the group algebra
of the associated group of the rack $X$
(see \cite{FennRourke} for definitions)
is obtained by localisation of $U(C)$ at
all group-likes. 
For the rack bialgebra $C$ with
group-like basis 
$1,g$ and $g \lhd g=1$, $U(C)$ is
the bialgebra with one group-like
generator $g$ satisfying $g^2=g$, 
so $U(C) \cong k \oplus k$ as algebra. 
If $U(C)$ does admit
an antipode, then at least over an
algebraically closed field of
characteristic 0 it is as a Hopf algebra 
isomorphic to a semidirect product of a
group algebra and a universal enveloping
algebra of a Lie algebra (see
e.g.~\cite[Theorem~3.8.2]{C}).  
Note further that there is an omission  
in \cite[Lemma 4.8]{KW} as the last
statement only makes sense when $H$ is a
Hopf algebra. 
\end{beispiel} 

\begin{beispiel}\label{beispiel_Leibniz}
If $C= k \cdot 1 \oplus  \mathfrak{h} $
is the rack bialgebra associated to a
Leibniz algebra $ \mathfrak {h}$ as in 
Example~\ref{example_Leibniz_algebra},
then $U(C)$ is the universal enveloping
algebra of the Lie algebra
$\mathfrak{h}_\mathrm{Lie}$ associated   
to $ \mathfrak{h}$, that is, the
quotient by the Leibniz ideal generated
by all squares $[x,x]$. Indeed, 
the generators of the ideal $J$ are in
this case of the form
$$
	[x,y] + y.x -x.y,
$$
and $J$ contains in particular all squares
$[x,x]$. The bialgebra object in 
$\mathcal{LM}$ obtained in
Theorem~\ref{theorem_bialgebra_object} is the
universal enveloping algebra of
the Lie algebra object 
$(\mathfrak{h},\mathfrak{h}_\mathrm{Lie})$ 
as in
\cite[Definition~4.3]{LP}. Thus
Theorem~\ref{theorem_bialgebra_object}
extends the construction of the
universal enveloping algebra of a
Leibniz algebra. 
\end{beispiel}

\begin{beispiel}
Let $ \mathfrak{g} $ be a Lie algebra, 
$H=U(\mathfrak{g} )$ be its
universal enveloping algebra, and 
$C \subset H$ be the image of 
$ k \oplus \mathfrak{g} \oplus
\mathfrak{g} \otimes \mathfrak{g} $,
that is, the degree 2 part in
the PBW filtration. This is a rack
bialgebra following the construction in
Proposition~\ref{husten} (starting with 
$\check C \subset \mathrm{ker}\,
\epsilon_H$). The symmetric algebra
$U(C)$ is not $H$, but 
$U( \mathfrak {g} \oplus S^2
\mathfrak{g})$, where $S^2 \mathfrak{g}
$ are the symmetric 2-tensors over $
\mathfrak{g} $, viewed as abelian Lie
algebra, and the direct sum is a direct
sum of Lie algebras. The kernel of 
$u \colon U(C) \rightarrow H$ 
is the ideal generated by $ S^2
\mathfrak{g} $. 
\end{beispiel}

\begin{bemerkung}
Before we continue, let us point out
that $U(C)$ differs from the Nichols
algebra associated to the braided vector
space $(\check C,\tau)$. The latter can
also be defined as a quotient algebra of
$T$, but with homogeneous relations in
degrees that can be of arbitrary degree,
cf.~\cite{V} for a pedagogical
introduction and original references. 
In contrast, the generators of $J$ are
in general inhomogeneous involving terms
of degree two and one. They are
homogeneous if and only if $ \lhd $
vanishes. In this case $ \tau $ is the
tensor flip and $ U(C)$ is the classical
symmetric algebra of the vector space
$\check C$. This is the only case in
which $U(C)$ agrees with the Nichols
algebra. 
\end{bemerkung}

\section{A non-cocommutative example}
\label{subsection_example}
Up to now, all examples of rack
bialgebras were cocommutative, and this
was an essential assumption in our
results. 
Also in our main reference \cite{CCES},
all examples were cocommutative (note
that the examples in Lemma~3.8 and
Lemma~3.9 therein are isomorphic
to each other). 
In this section, we present a
non-cocommutative example of a rack
bialgebra that nevertheless admits 
the structure of a Yetter-Drinfel'd rack
and can be constructed from a bialgebra object
in ${\mathcal L}{\mathcal M}$. 

\begin{proposition}\label{beispielnc}
Let 
$C={\rm Vect}(1,x,y,z,t)$ be the
coalgebra in which  
$t,y,z$ are primitives and 
$
	\triangle(x)\,=\,
	1 \otimes x+ x\otimes 1 + y\otimes z.
$
Then $C$ carries a unique rack bialgebra 
structure in which   
$-\lhd x,- \lhd t \colon C\to C$
are zero and 
\begin{align*}
& x\lhd z=t,\quad 
	x\lhd y=t,\quad 
	z\lhd z=0,\quad
	z\lhd y=0,\\ 
& y\lhd z=0,\quad
	y\lhd y=0,\quad 
	t\lhd z=0,\quad 
	t\lhd y=0. 
\end{align*}
\end{proposition}
\begin{proof}
The self-distributivity 
$$ 
	(a \lhd b) \lhd c \,=\, (a \lhd c_{(1)}) \lhd (b \lhd c_{(2)})
$$
is clear if one of the three elements is 1.
Otherwise, it follows from the fact that
$ (a \lhd b) \lhd c=0$ for 
$a,b,c \in \check C$, as then both sides
of the equation vanish.

Now let us check that
$$
	(a \lhd b)_{(1)} \otimes (a \lhd b)_{(2)} 
	= (a_{(1)} \lhd b_{(1)}) \otimes (a_{(2)} 
	\lhd b_{(2)}).
$$
For $b=1$, both sides are equal. 
For primitive $b$, the equation reads: 
$$(a \lhd b)_{(1)} \otimes (a \lhd b)_{(2)} = (a_{(1)} \lhd b) \otimes a_{(2)} 
+a_{(1)} \otimes (a_{(2)} \lhd b).$$
For $b=t$, both sides are clearly zero.
For $b=y$ or $b=z$, the only non-trivial case is $a=x$. We have for $a=x$ and $b=y$ for the LHS: 
$$
	(x \lhd y)_{(1)} \otimes (x \lhd y)_{(2)} = 
	t_{(1)}\otimes t_{(2)} = 
	1 \otimes t+ t\otimes 1.
$$
We have for $a=x$ and $b=y$ for the RHS:
\begin{eqnarray*}
	(x_{(1)} \lhd y) \otimes x_{(2)} +
	x_{(1)} \otimes (x_{(2)} \lhd y) 
&=& (1 \lhd y) \otimes x +
	(x \lhd y) \otimes 1 \\
&+& 1 \otimes (x \lhd y) + 
	x \otimes (1 \lhd y) \\
&+& (y \lhd y) \otimes z + 
	y \otimes (z \lhd y) \\
&=& t\otimes 1 + 1 \otimes t.
\end{eqnarray*} 
The case $a=x$ and $b=z$ is similar. The last case is the case $b=x$. 
The equation reads then:
$$ 0 = a_{(1)} \otimes (a_{(2)} \lhd x) +
(a_{(1)} \lhd x) \otimes a_{(2)} + (a_{(1)} \lhd y) \otimes (a_{(2)} \lhd z).$$
The first two terms (and the LHS) are zero, 
because $-\lhd x$ is zero. Concerning the third
term, the only case where it is non-zero is 
when $a_{(1)} \otimes a_{(2)}$ contains
a non-zero component proportional to 
$ x \otimes x$. However, there is no 
$a$ with this property in $C$. 
\end{proof}

Since $C$ is not cocommutative,
Theorem~\ref{theorem_braided_Leibniz}
can not be applied to construct a
canonical Yetter-Drinfel'd rack
structure. However, $C$ is a
Yetter-Drinfel'd rack over the
coordinate ring of the upper triangular
unipotent group in $GL(3)$:

\begin{proposition}
Let $H$ be the Hopf algebra whose
underlying algebra is the polynomial ring 
$k[X,Y,Z]$ with the coproduct 
$$
	\triangle(X) = 
	1 \otimes X + X \otimes 1 + Y \otimes
	Z,
$$
$$
	\triangle (Y) = 
	1 \otimes Y + Y \otimes 1,\quad
	\triangle (Z) = 
	1 \otimes Z + Z \otimes 1.
$$ 
Then the rack bialgebra $C$ from
Proposition~\ref{beispielnc} becomes a
Yetter-Drinfel'd rack over $H$ with 
$q(x)=X,q(y)=Y,q(z)=Z,q(t)=0$.  
\end{proposition}

\begin{proof}
The map $q$ is evidently a morphism of
coalgebras as $k \cdot t$ is a coideal. 
As the linear maps 
$ - \lhd x,- \lhd y,- \lhd z \colon C
\rightarrow C$ commute with each other,
there is a
well-defined right $H$-module structure 
on $C$ such that $a \lhd b = a \cdot
q(b)$. The $H$-module coalgebra
condition 
$(c \cdot h)_{(1)} \otimes (c \cdot h
)_{(2)} =  
(c_{(1)} \cdot h_{(1)}) \otimes 
(c_{(2)} \cdot h_{(2)})$   
and (\ref{d}) 
is verified by direct
computation when $h$ is one of the
generators $X,Y,Z$ and hence holds for
all $h \in H$.  
\end{proof}

Thus $C$ can be constructed as in
Corollary~\ref{schnupfen} inside the
Hopf algebra $T(\check C) \rtimes H$ 
despite the fact that it is
non-cocommutative, and as well from the
corresponding Hopf algebra object in
$\mathcal{LM}$.   

\section{Deformation cohomology}

In \cite{CCES}, the authors define 
cohomology groups controlling the
deformation theory of linear shelves. The method is to define the operations (i.e. the coproduct and the shelf product) on 
$C[[t]]$ instead of $C$, for a formal
parameter $t$, and then to impose
self-distributivity and the coalgebra
morphism condition on $\lhd$ as well as
the coassociativity of $ \triangle$. 
These requirements up to $t^{n+1}$ then give 
cocycle identities up to $t^{n+2}$. 
They are realised in a bicomplex whose
differentials are given explicitly 
up to degree 3, see Section~6 in
\cite{CCES} for details. 

On the other hand, in \cite{ABRWI} a
deformation complex for {\it
cocommutative} rack bialgebras $C$ is
defined. Therein, the deformations
involve only the shelf product, and not
the underlying coalgebra. Cochains are
defined to be coderivations with respect
to iterates of the shelf product. In
\cite{ABRW}, \cite{ABRWI}, left shelves
and Leibniz algebras are considered, so
we transpose the definitions here to
right shelves and Leibniz algebras.

Let $C$ be a rack bialgebra with a
cocommutative underlying coalgebra. Then
the rack product
$\mu(x,y):=x\lhd y$ can be iterated to
$$
	\mu^n(x_{1},\ldots,x_n):=(\ldots(x_{1}\lhd x_{2})\lhd\ldots)\lhd x_n,
$$
with the convention that $\mu^1=\id$ and $\mu^2=\mu$. 

Note that in the following definition we
moved the Sweedler notation to the top
to avoid confusion with the other
indices. 
   
\begin{definition}
Let $C$ be a rack bialgebra with a cocommutative underlying coalgebra. 
The deformation complex of $C$ is the graded vector space $C^*(C;C)$ defined in 
degree $n$ by
\[
C^n(C;C)\,:=\,\text{\rm Coder}(C^{\otimes n},C,\mu^n)
\]
denoting the space of coderivations along $\mu^n$, i.e. of linear maps $\omega:C^{\otimes n}\to C$ such that
$$\triangle_C\circ\omega=(\omega\otimes\mu^n+\mu^n\otimes\omega)\circ\triangle_{C^{\otimes n}},$$
endowed with the differential $d_C:C^*(C;C)\to C^{*+1}(C;C)$ defined in degree $n$ by
\[
d_C^n:=\sum_{i=1}^{n} (-1)^{i+1}(d_{i,1}^n-d_{i,0}^n)\:+\:(-1)^{n+1}d^n_{n+1}
\]
where the maps $d_{i,1}^n$ and $d_{i,0}^n$ are defined respectively by
\begin{multline*}
d_{i,1}^n\omega (r_{1},\ldots,r_{n+1}):=\\
\omega(r_{1},\ldots,r_{i-1},r_{i+1}^{(1)},\ldots,r_{n+1}^{(1)})\lhd
\mu^{n-i+2}(r_i,r_{i+1}^{(2)},\ldots,r_{n+1}^{(2)})
\end{multline*}
and
\[
d_{j,0}^n\omega (r_{1},\ldots,r_{n+1}):= 
\omega(r_{1}\lhd r_j^{(1)},\ldots,r_{j-1}\lhd r_j^{(j-1)},r_{j+1},\ldots,r_{n+1})
\]
and $d_{n+1}^n$ by
\begin{eqnarray*}
\lefteqn{  d_{n+1}^n\omega (r_{1},\ldots,r_{n+1})  }  \\ &:=& 
\mu^n(r_{1},r_{3}^{(1)}\ldots,r_{n+1}^{(1)})\lhd\omega(r_{2},r_{3}^{(2)},\ldots,r_{n+1}^{(2)})
\end{eqnarray*}
for all $\omega$ in $C^n(C;C)$ and $r_{1},\ldots,r_{n+1}$ in $C$. 
\end{definition}

It is shown in \cite{ABRWI} that $d_C^n\circ d_C^{n-1}=0$ and that $d_C^n$ sends coderivations to coderivations. 

\begin{bemerkung}
Cochains in the cohomology in Section 6 of \cite{CCES} are maps $C^{\otimes i}\to C^{\otimes j}$, like in the Gerstenhaber-Schack cohomology of associative bialgebras. Defining {\it special cochains} as 
$(\eta_{1},0,\ldots,0)$ where $\eta_{1}:C^{\otimes n}\to C$, one obtains that cocycles in $C^*(C,C)$
as above give rise to special cocycles. Indeed, while the first cocycle identity in \cite{CCES} is just the cocycle identity with respect to the above coboundary operator $d_C^n$, the second cocycle identity 
is the coderivation property and all other identities are trivial.  
\end{bemerkung}

\begin{beispiel}
The rack bialgebra $C={\rm Vect}(1,x,y,z,t)$ defined in Section \ref{subsection_example}
is a first order deformation in the
sense of the cohomology defined in
\cite{CCES} of the cocommutative shelf
in coalgebras $C_{0}={\rm Vect}(1,x,y,z,t)$
where $x,y,z,t$ are primitives and 
$1$ is group-like. 
The cocycle associated to the deformation is $\omega:C\to C^{\otimes 2}$ given by
$$\omega(x)=y\otimes z$$
and is trivial on the other basis
elements. Clearly, $C$ is not a deformation in the sense of the cohomology defined in \cite{ABRWI} as in this complex, the coproduct is not deformed. 
\end{beispiel}

We come to the main theorem of this section:

\begin{theorem}
Consider the rack bialgebra $C=k\oplus{\mathfrak h}$ associated to a (right) Leibniz algebra 
${\mathfrak h}$, see Example \ref{example_Leibniz_algebra}.
The Leibniz cohomology complex with values in the adjoint representation embeds into
the deformation complex $(C^*(C;C),d_C^*)$ defined above (by \cite{ABRWI}). 
\end{theorem}

\begin{proof}
We extend Leibniz cochains $f:{\mathfrak h}^{\otimes n}\to{\mathfrak h}$ to cochains in the complex
$C^*(C,C)$ with $C=k \cdot 1
\oplus{\mathfrak h}$ by setting them
zero on all components in $k \cdot 1 \subset C$. 
More precisely
$$
	\omega((\lambda_{1},x_{1}),\ldots,(\lambda_n,x_n)):={\rm pr}_{\mathfrak h}(f(x_{1},\ldots,x_n)),
$$
for $(\lambda_i,x_i)\in k \cdot 1 \oplus{\mathfrak h}$ for all $i=1,\ldots,n$ and with 
${\rm pr}_{\mathfrak h}:k \cdot 1 \oplus{\mathfrak h}\to{\mathfrak h}$ the natural projection.   

With this definition, it follows that these cochains are coderivations along $\mu^n$, i.e.
$$\triangle_C\circ\omega=(\omega\otimes\mu^n+\mu^n\otimes\omega)\circ\triangle_{C^{\otimes n}}.$$
Indeed, when computing the iterated coproduct $\triangle_{C^{\otimes n}}(r_{1},\ldots,r_n)$, the elements 
$r_i\in{\mathfrak h}$ are distributed
among the two factors in $C^{\otimes
n}\otimes C^{\otimes n}$ and all other
components are filled with units. On the LHS, $\omega(r_{1},\ldots,r_n)$ is primitive by construction,
thus we get the two terms $\omega(r_{1},\ldots,r_n)\otimes 1+1\otimes\omega(r_{1},\ldots,r_n)$.
On the RHS, the only terms which do not
vanish are those with all $r_i$ as arguments in $\omega$.
This shows the equality. 

Now we specify the different parts of the coboundary operator.  

\begin{multline*}
d_{i,1}^n\omega (r_{1},\ldots,r_{n+1})=\\
\omega(r_{1},\ldots,r_{i-1},r_{i+1}^{(1)},\ldots,r_{n+1}^{(1)})\lhd
\mu^{n-i+2}(r_i,r_{i+1}^{(2)},\ldots,r_{n+1}^{(2)})    \\
=[\omega(r_{1},\ldots,r_{i-1},r_{i+1},\ldots,r_{n+1}),r_i],
\end{multline*}
because the only contributing term is the one where all $r_j$ are arguments of $\omega$, i.e. 
all the units are in $\mu^{n-i+2}$. 

\begin{eqnarray*}
d_{j,0}^n\omega (r_{1},\ldots,r_{n+1})&=& 
\omega(r_{1}\lhd r_j^{(1)},\ldots,r_{j-1}\lhd r_j^{(j-1)},r_{j+1},\ldots,r_{n+1})  \\
&=&\omega([r_{1},r_j],\ldots,r_{j-1},r_{j+1},\ldots,r_{n+1}) +\ldots \\
\ldots&+&\omega(r_{1},\ldots,[r_{j-1},r_j],r_{j+1},\ldots,r_{n+1}),  
\end{eqnarray*}
because only one of the $r_j^{(k)}$ is equal to $r_j$ and all others are equal to $1$. 

\begin{eqnarray*}
\lefteqn{  d_{n+1}^n\omega (r_{1},\ldots,r_{n+1})  }  \\ &=& 
\mu^n(r_{1},r_{3}^{(1)}\ldots,r_{n+1}^{(1)})\lhd\omega(r_{2},r_{3}^{(2)},\ldots,r_{n+1}^{(2)})  \\
&=&[r_{1},\omega(r_{2},r_{3},\ldots,r_{n+1})],
\end{eqnarray*}
because this is the only term where one does not insert $1$ into $\omega$. 
\end{proof} 

\begin{bemerkung}
A similar statement is true for the rack bialgebra $C$ associated to a (set-theoretical) shelf, cf Example
\ref{example_shelf}.
In fact, the deformation complex in \cite{ABRWI} has been constructed as a linearization of the cohomology complex of a shelf. 
\end{bemerkung}

\section{Outlook and further questions}
One natural direction of further
research is to link the three different 
approaches to deformation theories of 
rack bialgebras:
\begin{enumerate}
\item[(i)] The bicomplex of 
Carter-Crans-Elhamdadi-Saito \cite{CCES}.
\item[(ii)] The Gerstenhaber-Schack
bicomplex in ${\mathcal L}{\mathcal M}$
\cite{GS}. 
\item[(iii)] The braided Leibniz 
complex \cite{Leb}. 
\end{enumerate}

Our previous article \cite{KW} showed
how objects in the setting (ii) can be
expressed as objects in the setting
(iii). The present article provides the
link from (i) to (ii) and from (i) to 
(iii) in the cocommutative case. 
A deformation theory for cocommutative
rack bialgebras alone should however
rather be built on the
dual version of Andr\'e-Quillen cohomology 
(the cohomology theory that controls
deformations of commutative algebras)
than the Cartier cohomology that
underlies Gerstenhaber-Schack
cohomology. 

In general, Gerstenhaber-Schack cohomology 
of a bialgebra $f \colon M \rightarrow
H$ in ${\mathcal L}{\mathcal M}$ 
captures more information than the
cohomology from \cite{CCES}, as one can
deform $H,M$ and $f$. On
the other hand, braided Leibniz
cohomology seems to capture less
information than \cite{CCES}, because it 
only indirectly reflects deformations of the 
coproduct. 
A full investigation of the relation
between the three settings seems a
fruitful future research direction.

Two rather concrete questions that arises
from this article are:

\begin{question}
Is there a (necessarily
non-cocommutative) rack bialgebra
that can not be expressed as a
Yetter-Drinfel'd rack over any
bialgebra?
\end{question}

\begin{question}
For which rack bialgebras is $U(C)$ a
Hopf algebra? 
\end{question}

This matters in the passage from 
(ii) to (i) as the antipode is in
general necessary to turn the 
coinvariants in a Hopf tetramodule into
a right module. One expects that $U$ is
then part of an adjoint pair of functors
between such Hopf racks and Hopf
algebras, at least in the cocommutative
setting.

\end{document}